\documentclass[amstex,12pt, amssymb]{article}

\usepackage{mathtext}
\usepackage[cp1251]{inputenc}
\usepackage[T2A]{fontenc}
\usepackage[dvips]{graphicx}
\usepackage{amsmath}
\usepackage{amssymb}
\usepackage{amsxtra}
\usepackage{latexsym}
\usepackage{ifthen}

\newcounter{lemma}[section]

\newcounter{corollary}[section]

\newcounter{remark}[section]

\newcounter{theorem}[section]

\newcounter{proposition}[section]

\newcounter{example}

\numberwithin{equation}{section}

\begin{document}

\markboth{\centerline{E.~SEVOST'YANOV}}{\centerline{ON  LOCAL
BEHAVIOR OF MAPPINGS WITH INTEGRAL CONSTRAINTS...}}

\def\cc{\setcounter{equation}{0}
\setcounter{figure}{0}\setcounter{table}{0}}

\overfullrule=0pt


\author{E.~SEVOST'YANOV}

\title{
{\bf ON GLOBAL BEHAVIOR OF MAPPINGS WITH INTEGRAL CONSTRAINTS}}

\date{\today}
\maketitle

\begin{abstract}
This article is devoted to the study of mappings with branch points
whose characteristics satisfy integral-type constraints. We have
proved theorems concerning their local and global behavior.  In
particular, we established the equicontinuity of families of such
mappings inside their definition domain, as well as, under
additional conditions, equicontinuity of the families of these
mappings in its closure.
\end{abstract}

\bigskip
{\bf 2010 Mathematics Subject Classification: Primary 30C65;
Secondary 31A15, 31B25}

\section{Introduction}

This article is devoted to the study of mappings satisfying upper
bounds for the distortion of the modulus of families of paths, see,
for example, \cite{Cr}, \cite{GRY}, \cite{MRV$_1$}--\cite{MRV$_2$},
\cite{MRSY} and \cite{RSY}. In particular, here we are dealing with
mappings whose characteristics satisfy the so-called conditions of
integral type, see, for example, \cite{RSY}, \cite{RS} and
\cite{Sev$_3$}. The main purpose of the present manuscript is to
study the local behavior of mappings having branch points whose
characteristics are bounded only on the average. It is worth noting
some of the previous results in this direction. In particular,
in~\cite{RS}, homeomorphisms with similar conditions were
investigated, and in~\cite{Sev$_3$}, mappings with branch points
having a general characteristic $Q.$ Unfortunately, the most general
case when the mappings are not homeomorphisms and also do not have a
common majorant has been overlooked. Note that the most interesting
applications related to the study of the Beltrami equation and the
Dirichlet problem for it are associated with the absence of a
general majorant for maximal complex characteristics, (see, e.g.,
\cite{Dyb} and \cite{L}).

\medskip
Let us move on to definitions and formulations of results. Given
$p\geqslant 1,$ $M_p$ denotes the $p$-modulus of a family of paths,
and the element $dm(x)$ corresponds to a Lebesgue measure in ${\Bbb
R}^n,$ $n\geqslant 2,$ see~\cite{Va}. In what follows, we usually
write $M(\Gamma)$ instead of $M_n(\Gamma).$ For the sets $A,
B\subset{\Bbb R}^n$ we set, as usual,
$${\rm diam}\,A=\sup\limits_{x, y\in A}|x-y|\,,\quad {\rm dist}\,(A, B)=\inf\limits_{x\in A,
y\in B}|x-y|\,.$$
Sometimes we also write $d(A)$ instead of ${\rm diam}\,A$ and $d(A,
B)$ instead of ${\rm dist\,}(A, B),$ if no misunderstanding is
possible. For given sets $E$ and $F$ and a given domain $D$ in
$\overline{{\Bbb R}^n}={\Bbb R}^n\cup \{\infty\},$ we denote by
$\Gamma(E, F, D)$ the family of all paths $\gamma:[0, 1]\rightarrow
\overline{{\Bbb R}^n}$ joining $E$ and $F$ in $D,$ that is,
$\gamma(0)\in E,$ $\gamma(1)\in F$ and $\gamma(t)\in D$ for all
$t\in (0, 1).$ Everywhere below, unless otherwise stated, the
boundary and the closure of a set are understood in the sense of an
extended Euclidean space $\overline{{\Bbb R}^n}.$ Let
$x_0\in\overline{D},$ $x_0\ne\infty,$
$$S(x_0,r) = \{
x\,\in\,{\Bbb R}^n : |x-x_0|=r\}\,, S_i=S(x_0, r_i)\,,\quad
i=1,2\,,$$
\begin{equation}\label{eq6} A=A(x_0, r_1, r_2)=\{ x\,\in\,{\Bbb R}^n :
r_1<|x-x_0|<r_2\}\,.
\end{equation}
Everywhere below, unless otherwise stated, the closure
$\overline{A}$ and the boundary $\partial A $ of the set $A$ are
understood in the topology of the space $\overline{{\Bbb R}^n}={\Bbb
R}^n\cup\{\infty\}.$ Let $Q:{\Bbb R}^n\rightarrow {\Bbb R}^n$ be a
Lebesgue measurable function satisfying the condition $Q(x)\equiv 0$
for $x\in{\Bbb R}^n\setminus D.$ Given $p\geqslant 1,$ a mapping
$f:D\rightarrow \overline{{\Bbb R}^n}$ is called a {\it ring
$Q$-mapping at the point $x_0\in \overline{D}\setminus \{\infty\}$
with respect to $p$-modulus}, if the condition
\begin{equation} \label{eq2*!}
M_p(f(\Gamma(S_1, S_2, D)))\leqslant \int\limits_{A\cap D} Q(x)\cdot
\eta^p (|x-x_0|)\, dm(x)
\end{equation}
holds for all $0<r_1<r_2<d_0:=\sup\limits_{x\in D}|x-x_0|$ and all
Lebesgue measurable functions $\eta:(r_1, r_2)\rightarrow [0,
\infty]$ such that
\begin{equation}\label{eq8B}
\int\limits_{r_1}^{r_2}\eta(r)\,dr\geqslant 1\,.
\end{equation}
Inequalities of the form~(\ref{eq2*!}) were established for many
well-known classes of mappings. So, for quasiconformal mappings and
mappings with bounded distortion, they hold for $p=n$ and some
$Q(x)\equiv K=const$ (see, for example, \cite[Theorem~7.1]{MRV$_1$}
and \cite[Definition~13.1]{Va}). Such inequalities also hold for
many mappings with unbounded characteristic, in particular, for
homeomorphisms belonging to the class $W_{\rm loc}^{1, p},$ $p>n-1,$
the inner dilatation of the order $\alpha:=\frac{p}{p-n+1}$ is
locally integrable (see, for example, \cite[Theorems~8.1, 8.5]{MRSY}
and \cite[Corollary~2]{Sal$_1$}, \cite[Theorem~9,
Lemma~5]{Sal$_2$}).

\medskip
The concept of a set of capacity zero, used below, can be found
in~\cite[Section~2.12]{MRV$_2$} and is therefore omitted. A mapping
$f:D\rightarrow {\Bbb R}^n$ is called {\it discrete} if the preimage
$\{f^{\,-1}(y)\}$ of each point $y\,\in\,{\Bbb R}^n$ consist of
isolated points, and {\it open} if the image of any open set
$U\subset D $ is an open set in ${\Bbb R}^n.$ Let us formulate the
main results of this manuscript.
In what follows, $h$ denotes the so-called chordal metric defined by
the equalities
\begin{equation}\label{eq1E}
h(x,y)=\frac{|x-y|}{\sqrt{1+{|x|}^2} \sqrt{1+{|y|}^2}}\,,\quad x\ne
\infty\ne y\,, \quad\,h(x,\infty)=\frac{1}{\sqrt{1+{|x|}^2}}\,.
\end{equation}
For a given set $E\subset\overline{{\Bbb R}^n},$ we set
\begin{equation}\label{eq9C}
h(E):=\sup\limits_{x,y\in E}h(x, y)\,,
\end{equation}
The quantity $h(E)$ in~(\ref{eq9C}) is called the {\it chordal
diameter} of the set $E.$ For given sets $A, B\subset
\overline{{\Bbb R}^n},$ we put
$h(A, B)=\inf\limits_{x\in A, y\in B}h(x, y),$
where $h$ is a chordal metric defined in~(\ref{eq1E}).

\medskip
Given a domain $D\subset {\Bbb R}^n,$ a number $M_0>0,$ a set
$E\subset \overline{{\Bbb R}^n}$ and a strictly increasing function
$\Phi\colon\overline{{\Bbb R}^{+}}\rightarrow\overline{{\Bbb
R}^{+}}$ let us denote by ${\frak F}^{\Phi}_{M_0, E}(D)$ the family
of all open discrete mappings $f:D\rightarrow \overline{{\Bbb
R}^n}\setminus E$ for which there exists a function
$Q=Q_f(x):D\rightarrow [0, \infty]$ such
that~(\ref{eq2*!})--(\ref{eq8B}) hold for any $x_0\in D$ with $p=n$
and, in addition,
\begin{equation}\label{eq2!!A}
\int\limits_D\Phi(Q(x))\frac{dm(x)}{\left(1+|x|^2\right)^n}\
\leqslant M_0<\infty\,.
\end{equation}
An analogue of the following statement was established for
homeomorphisms in~\cite[Theorem~4.1]{RS}, and for mappings whose
corresponding function $Q$ is fixed, in~\cite[Theorem~1]{Sev$_3$}.
However, let us note that it is in the form given below that the
indicated statement seems to be the most interesting from the point
of view of applications to the problem of compactness of solutions
of the Beltrami equations and the Dirichlet problem (see, for
example, \cite[Theorem~2]{Dyb}) and \cite[Theorem~1]{L}).

\medskip
\begin{theorem}\label{th1}
{\sl\, Let $D$ be a domain in ${\Bbb R}^n,$ $n\geqslant 2,$ and let
${\rm cap}\,E>0.$ If
\begin{equation}\label{eq3!A} \int\limits_{\delta_0}^{\infty}
\frac{d\tau}{\tau\left[\Phi^{-1}(\tau)\right]^{\frac{1}{n-1}}}=
\infty
\end{equation}
holds for some $\delta_0>\tau_0:=\Phi(0),$ then ${\frak
F}^{\Phi}_{M_0, E}(D)$ is equicontinuous in $D.$
 }
\end{theorem}

\medskip
Note that the statement of Theorem~\ref{th1} is much simpler and
more elegant for the case $p\in(n-1, n).$ Given $p\geqslant 1,$ a
domain $D\subset {\Bbb R}^n,$ a number $M_0>0$ and a strictly
increasing function $\Phi\colon\overline{{\Bbb
R}^{+}}\rightarrow\overline{{\Bbb R}^{+}}$ let us denote by ${\frak
F}^{\Phi}_{M_0, p}(D)$ the family of all open discrete mappings
$f:D\rightarrow {\Bbb R}^n$ for which there exists a function
$Q=Q_f(x):D\rightarrow [0, \infty]$ such
that~(\ref{eq2*!})--(\ref{eq8B}) hold for any $x_0\in D$ and, in
addition, (\ref{eq2!!A}) holds. The following statement is true.

\medskip
\begin{theorem}\label{th2}
{\sl\, Let $D$ be a domain in ${\Bbb R}^n,$ $n\geqslant 2,$ and let
$p\in (n-1, n).$ If~(\ref{eq3!A}) holds for some
$\delta_0>\tau_0:=\Phi(0),$ then ${\frak F}^{\Phi}_{M_0, p}(D)$ is
equicontinuous in $D.$
 }
\end{theorem}

\medskip
\begin{remark}
Let $(X, d)$ and $\left(X^{{\,\prime}}, d^{\,\prime}\right)$ be
metric spaces with distances $d$ and $d^{\,\prime},$ respectively. A
family $\frak{F}$ of mappings $f:X\rightarrow X^{\,\prime}$ is said
to be {\it equicontinuous at a point} $x_0\in X,$ if for every
$\varepsilon>0$ there is $\delta=\delta(\varepsilon, x_0)>0$ such
that $d^{\,\prime}(f(x), f(x_0))<\varepsilon$ for all $f\in
\frak{F}$ and $x\in X$ with $d(x , x_0)<\delta$. The family
$\frak{F}$ is {\it equicontinuous} if $\frak{F}$ is equicontinuous
at every point $x_0\in X.$

\medskip
In Theorem~\ref{th1}, the equicontinuity of the corresponding family
of mappings should be understood in the sense of mappings of the
metric spaces $(X, d)$ and $\left(X^{\,\prime}, d^{\,\prime}
\right),$ where $X$ is a domain $D,$ and $d$ is a Euclidean metric
in $D,$ besides, $X^{\,\prime}=\overline{{\Bbb R}^n},$ and $h$ is a
chordal metric defined in~(\ref{eq1E}). At the same time, in
Theorem~\ref{th2} the space $X$ remains the same, and the space
$X^{\,\prime}$ is an usual Euclidean $n$-dimensional space with the
Euclidean metric $d^{\,\prime}.$
\end{remark}

\medskip
A separate research topic is the equicontinuity of families of
mappings in the closure of a domain. Results of this kind for fixed
characteristics were obtained in some of our papers.  In particular,
in~\cite{Sev$_4$} we considered the case of fixed domains between
which the mappings act, and in the
papers~\cite{SevSkv$_1$}--\cite{SevSkv$_2$} we considered the case
when the mapped domain can change. It should be noted that theorems
on normal families of mappings are especially important in the study
of the properties of solutions to the Dirichlet problem for the
Beltrami equation (see, for example, \cite{Dyb}). Note that the
classical results on the equicontinuity of quasiconformal mappings
in the closure of a domain were obtained by N\"{a}kki and Palka, see
e.g.~\cite[Theorem~3.3]{NP}. Let us formulate the main results
related to this case.

\medskip
Let $I$ be a fixed set of indices and let $D_i,$ $i\in I,$ be some
sequence of domains. Following~\cite[Sect.~2.4]{NP}, we say that a
family of domains $\{D_i\}_{i\in I}$ is {\it equi-uniform with
respect to $p$-modulus} if for any $r> 0$ there exists a number
$\delta> 0$ such that the inequality
\begin{equation}\label{eq17***}
M_p(\Gamma(F^{\,*},F, D_i))\geqslant \delta
\end{equation}
holds for any $i\in I$ and any continua $F, F^*\subset D$ such that
$h(F)\geqslant r$ and $h(F^{\,*})\geqslant r.$ It should be noted
that the condition of equi-uniformity of the sequence of domains
implies strong accessibility of the boundary of each of them with
respect to $p$-modulus (see, for example,
\cite[Remark~1]{SevSkv$_1$}). Given $p\geqslant 1,$ u number
$\delta>0,$ a domain $D\subset {\Bbb R}^n,$ $n\geqslant 2,$ a
continuum $A\subset D$ and a strictly increasing function
$\Phi\colon\overline{{\Bbb R}^{+}}\rightarrow\overline{{\Bbb
R}^{+}}$ denote $\frak{F}_{\Phi, A, p, \delta}(D)$ the family of all
homeomorphisms $f:D\rightarrow {\Bbb R}^n$ for which there exists a
function $Q=Q_f(x):D\rightarrow [0, \infty]$ such that: 1) relations
(\ref{eq2*!})--(\ref{eq8B}) hold for any $x_0\in \overline{D},$ 2)
the relation (\ref{eq2!!A}) holds and 3) the relations
$h(f(A))\geqslant\delta$ and $h(\overline{{\Bbb R}^n}\setminus
f(D))\geqslant \delta$ hold. The following statement is true.

\medskip
\begin{theorem}\label{th3} {\sl\, Let $p\in (n-1, n],$ a domain $D$
is locally connected at any $x_0\in\partial D$ and domains
$D_f^{\,\prime}=f(D)$ are equi-uniform with respect to $p$-modulus
over all $f\in \frak{F}_{\Phi, A, p, \delta}(D).$ If~(\ref{eq3!A})
holds for some $\delta_0>\tau_0:=\Phi(0),$ then any
$f\in\frak{F}_{Q, A, p, \delta}(D)$ has a continuous extension in
$\overline{D}$ and, besides that, the family $\frak{F}_{\Phi, A, p,
\delta}(\overline{D})$ of all extended mappings $\overline{f}:
\overline{D}\rightarrow \overline{{\Bbb R}^n},$ is equicontinuous in
$\overline{D}.$
  }
\end{theorem}

\medskip
As usual, we use the notation
$$C(f, x):=\{y\in \overline{{\Bbb R}^n}:\exists\,x_k\in D: x_k\rightarrow x, f(x_k)
\rightarrow y, k\rightarrow\infty\}\,.$$
A mapping $f$ between domains $D$ and $D^{\,\prime}$ is called {\it
closed} if $f(E)$ is closed in $D^{\,\prime}$ for any closed set
$E\subset D$ (see, e.g., \cite[Section~3]{Vu}). Any open discrete
closed mapping is boundary preserving, i.e.  $C(f, \partial
D)\subset
\partial D^{\,\prime},$ where $C(f, \partial D)=\bigcup\limits_{x\in \partial D}C(f, x)$
(see e.g.~\cite[Theorem~3.3]{Vu}).

\medskip
Given $p\geqslant 1,$ a domain $D\subset {\Bbb R}^n,$ a set
$E\subset\overline{{\Bbb R}^n},$ a strictly increasing function
$\Phi\colon\overline{{\Bbb R}^{+}}\rightarrow\overline{{\Bbb
R}^{+}}$ and a number $\delta>0$ denote by $\frak{R}_{\Phi, \delta,
p, E}(D)$ the family of all open discrete and closed mappings
$f:D\rightarrow \overline{{\Bbb R}^n}\setminus E$ such that: 1)
relations (\ref{eq2*!})--(\ref{eq8B}) hold for any $x_0\in
\overline{D},$ 2) the relation (\ref{eq2!!A}) holds and 3) there
exists a continuum $K_f\subset D^{\,\prime}_f$ such that
$h(K_f)\geqslant \delta$ and $h(f^{\,-1}(K_f),
\partial D)\geqslant \delta>0.$ The following
statement is true.

\medskip
\begin{theorem}\label{th4} {\sl\, Let $p\in (n-1, n],$ a domain $D$
is locally connected at any point $x_0\in\partial D$ and, besides
that, domains $D_f^{\,\prime}=f(D)$ are equi-uniform with respect to
$p$-modulus over all $f\in\frak{R}_{\Phi, \delta, p, E}(D).$ Let
${\rm cap\,}E>0$ for $p=n,$ and let $E$ is any closed set for
$n-1<p<n.$ If~(\ref{eq3!A}) holds for some
$\delta_0>\tau_0:=\Phi(0),$ then any $f\in\frak{R}_{\Phi, \delta, p,
E}(D)$ has a continuous extension in $\overline{D}$ and, besides
that, the family $\frak{R}_{\Phi, \delta, p, E}(\overline{D})$ of
all extended mappings $\overline{f}: \overline{D}\rightarrow
\overline{{\Bbb R}^n},$ is equicontinuous in $\overline{D}.$}
\end{theorem}

\begin{remark}\label{rem3}
In Theorems~\ref{th3} and~\ref{th4}, the equicontinuity should be
understood in terms of families of mappings between metric
spaces~$(X, d)$ and $\left(X^{\,\prime}, d^{\,\prime}\right),$ where
$X=\overline{D},$ $d$ is a chordal metric $h,$
$X^{\,\prime}=\overline{{\Bbb R}^n}$ and $d^{\,\prime}$ is a chordal
(spherical) metric $h,$ as well.
\end{remark}

\medskip
Theorems~\ref{th3} and~\ref{th4} admit a natural generalization to
the case of complex boundaries, when the maps do not have a
continuous extension to points of the boundary of the domain in the
usual sense, however, this extension holds in the sense of the
so-called prime ends. Let us recall several important definitions
associated with this concept. In the following, the next notation is
used: the set of prime ends corresponding to the domain $D,$ is
denoted by $E_D,$ and the completion of the domain $D$ by its prime
ends is denoted $\overline{D}_P.$ The definition of prime ends used
below corresponds to the definition given in~\cite{IS$_2$}, and
therefore is omitted. Consider the following definition, which goes
back to N\"akki~\cite{Na$_2$}, see also~\cite{KR}. We say that the
boundary of the domain $D$ in ${\Bbb R}^n$ is {\it locally
quasiconformal}, if each point $x_0\in\partial D$ has a neighborhood
$U$ in ${\Bbb R}^n$, which can be mapped by a quasiconformal mapping
$\varphi$ onto the unit ball ${\Bbb B}^n\subset{\Bbb R}^n$ so that
$\varphi(\partial D\cap U)$ is the intersection of ${\Bbb B}^n$ with
the coordinate hyperplane.

\medskip For a given set $E\subset {\Bbb R}^n,$ we set
$d(E):=\sup\limits_{x, y\in E}|x-y|.$
The sequence of cuts $\sigma_m,$ $m=1,2,\ldots ,$ is called {\it
regular,} if
$\overline{\sigma_m}\cap\overline{\sigma_{m+1}}=\varnothing$ for
$m\in {\Bbb N}$ and, in addition, $d(\sigma_{m})\rightarrow 0$ as
$m\rightarrow\infty.$ If the end $K$ contains at least one regular
chain, then $K$ will be called {\it regular}. We say that a bounded
domain $D$ in ${\Bbb R}^n$ is {\it regular}, if $D$ can be
quasiconformally mapped to a domain with a locally quasiconformal
boundary whose closure is a compact in ${\Bbb R}^n,$ and, besides
that, every prime end in $D$ is regular. Note that space
$\overline{D}_P=D\cup E_D$ is metric, which can be demonstrated as
follows. If $g:D_0\rightarrow D$ is a quasiconformal mapping of a
domain $D_0$ with a locally quasiconformal boundary onto some domain
$D,$ then for $x, y\in \overline{D}_P$ we put:
\begin{equation}\label{eq5}
\rho(x, y):=|g^{\,-1}(x)-g^{\,-1}(y)|\,,
\end{equation}
where the element $g^{\,-1}(x),$ $x\in E_D,$ is to be understood as
some (single) boundary point of the domain $D_0.$ The specified
boundary point is unique and well-defined by~\cite[Theorem~2.1,
Remark~2.1]{IS$_2$}, cf.~\cite[Theorem~4.1]{Na$_2$}. It is easy to
verify that~$\rho$ in~(\ref{eq5}) is a metric on $\overline{D}_P,$
and that the topology on $\overline{D}_P,$ defined by such a method,
does not depend on the choice of the map $g$ with the indicated
property. The analogs of Theorems~\ref{th3} and~\ref{th4} for the
case of prime ends are as follows.

\medskip
\begin{theorem}\label{th5} {\sl\, Let $p\in (n-1, n]$ and let $D$ be a regular
domain.  Assume that $D_f^{\,\prime}=f(D)$ are bounded equi-uniform
domains with respect to $p$-modulus over all $f\in \frak{F}_{\Phi,
A, p, \delta}(D),$ which are domains with a locally quasiconformal
boundary, as well. If~(\ref{eq3!A}) holds for some
$\delta_0>\tau_0:=\Phi(0),$ then any $f\in\frak{F}_{Q, A, p,
\delta}(D)$ has a continuous extension in $\overline{D}_P$ and,
besides that, the family $\frak{F}_{\Phi, A, p,
\delta}(\overline{D})$ of all extended mappings $\overline{f}:
\overline{D}_P\rightarrow \overline{{\Bbb R}^n},$ is equicontinuous
in $\overline{D}_P.$ }
\end{theorem}

\medskip
\begin{theorem}\label{th6} {\sl\, Let $p\in (n-1, n]$ and let $D$ be a regular domain.
Assume that domains $D_f^{\,\prime}=f(D)$ are bounded equi-uniform
domains with respect to $p$-modulus over all $f\in\frak{R}_{\Phi,
\delta, p, E}(D),$ which are domains with a locally quasiconformal
boundary, as well. Let ${\rm cap\,}E>0$ for $p=n,$ and let $E$ is
any closed domain whenever $n-1<p<n.$ If~(\ref{eq3!A}) holds for
some $\delta_0>\tau_0:=\Phi(0),$ then any $f\in\frak{R}_{\Phi,
\delta, p, E}(D)$ has a continuous extension in $\overline{D}_P$
and, besides that, the family $f\in\frak{R}_{\Phi, \delta, p,
E}(\overline{D})$ of all extended mappings $\overline{f}:
\overline{D}_P\rightarrow \overline{{\Bbb R}^n},$ is equicontinuous
in $\overline{D}_P.$ }
\end{theorem}


\begin{remark}\label{rem2}
In Theorems~\ref{th5} and~\ref{th6}, the equicontinuity should be
understood in terms of families of mappings between metric
spaces~$(X, d)$ and $\left(X^{\,\prime}, d^{\,\prime}\right),$ where
$X=\overline{D}_P,$ $d$ is one of the possible metrics,
corresponding to the topological space $\overline{D}_P,$
$X^{\,\prime}=\overline{{\Bbb R}^n}$ and $d^{\,\prime}$ is a chordal
(spherical) metric.
\end{remark}

\section{Auxiliary Lemmas}

As in the article \cite{RS}, the key point related to the proof of
the main statements of the article is related to the connection
between conditions~(\ref{eq2!!A})-(\ref{eq3!A}) and the divergence
of an integral of a special form (see, for example,
\cite[Lemma~3.1]{RS}). Given a Lebesgue measurable function $Q:{\Bbb
R}^n\rightarrow [0, \infty]$ and a point $x_0\in {\Bbb R}^n$ we set
\begin{equation}\label{eq10}
q_{x_0}(t)=\frac{1}{\omega_{n-1}r^{n-1}} \int\limits_{S(x_0,
t)}Q(x)\,d\mathcal{H}^{n-1}\,,
\end{equation}
where $\mathcal{H}^{n-1}$ denotes $(n-1)$-dimensional Hausdorff
measure. The following lemma is of particular importance.

\medskip
\begin{lemma}\label{lem1}
{\sl\, Let $1\leqslant p\leqslant n,$ and let $\Phi:[0,
\infty]\rightarrow [0, \infty] $ be a strictly increasing convex
function such that the relation
\begin{equation}\label{eq2} \int\limits_{\delta_0}^{\infty}
\frac{d\tau}{\tau\left[\Phi^{-1}(\tau)\right]^{\frac{1}{p-1}}}=
\infty
\end{equation}
holds for some $\delta_0>\tau_0:=\Phi(0).$ Let $\frak{Q}$ be a
family of functions $Q:{\Bbb R}^n\rightarrow [0, \infty]$ such that
\begin{equation}\label{eq5A}
\int\limits_D\Phi(Q(x))\frac{dm(x)}{\left(1+|x|^2\right)^n}\
\leqslant M_0<\infty
\end{equation}
for some $0<M_0<\infty.$ Now, for any $0<r_0<1$ and for every
$\sigma>0$ there exists $0<r_*=r_*(\sigma, r_0, \Phi)<r_0$ such that
\begin{equation}\label{eq4}
\int\limits_{\varepsilon}^{r_0}\frac{dt}{t^{\frac{n-1}{p-1}}q^{\frac{1}{p-1}}_{x_0}(t)}\geqslant
\sigma\,,\qquad \varepsilon\in (0, r_*)\,,
\end{equation}
for any $Q\in \frak{Q}.$ }
\end{lemma}

\medskip
\begin{proof}
Using the substitution of variables $t=r/r_0, $ for any
$\varepsilon\in (0, r_0) $ we obtain that
\begin{equation}\label{eq34}
\int\limits_{\varepsilon}^{r_0}\frac{dr}{r^{\frac{n-1}{p-1}}q^{\frac{1}{p-1}}_{x_0}(r)}
\geqslant
\int\limits_{\varepsilon}^{r_0}\frac{dr}{rq^{\frac{1}{p-1}}_{x_0}(r)}
=\int\limits_{\varepsilon/r_0}^1\frac{dt}{tq^{\frac{1}{p-1}}_{x_0}(tr_0)}
=\int\limits_{\varepsilon/r_0}^1\frac{dt}{t\widetilde{q}^{\frac{1}{p-1}}_{0}(t)}\,,
\end{equation}
where $\widetilde{q}_0(t)$ is the average integral value of the
function $\widetilde{Q}(x):=Q(r_0x+x_0)$ over the sphere $|x|=t,$
see the ratio~(\ref{eq10}). Then, according to~\cite[Lemma~3.1]{RS},
\begin{equation}\label{eq35}
\int\limits_{\varepsilon/r_0}^1\frac{dt}{t\widetilde{q}^{\frac{1}{p-1}}_{0}(t)}\geqslant
\frac{1}{n}\int\limits_{eM_*\left(\varepsilon/r_0\right)}^{\frac{M_*\left(\varepsilon/r_0\right)
r_0^n}{\varepsilon^n}}\frac{d\tau}
{\tau\left[\Phi^{-1}(\tau)\right]^{\frac{1}{p-1}}}\,,
\end{equation}
where
$$M_*\left(\varepsilon/r_0\right)=
\frac{1}{\Omega_n\left(1-\left(\varepsilon/r_0\right)^n\right)}
\int\limits_{A\left(0, \varepsilon/r_0, 1\right)} \Phi\left(Q(r_0x
+x_0)\right)\,dm(x)=$$
$$=
\frac{1}{\Omega_n\left(r_0^n-\varepsilon^n\right)}
\int\limits_{A\left(x_0, \varepsilon, r_0\right)}
\Phi\left(Q(x)\right)\,dm(x)$$
and~$A(x_0, \varepsilon, r_0)$ is defined in~(\ref{eq6}) for
$r_1:=\varepsilon$ and $r_2:=r_0.$ Observe that $|x|\leqslant
|x-x_0|+ |x_0|\leqslant r_0+|x_0|$ for any~$x\in A(x_0, \varepsilon,
r_0).$ Thus
$$M_*\left(\varepsilon/r_0\right)\leqslant \frac{\beta(x_0)}
{\Omega_n\left(r_0^n-\varepsilon^n\right)}\int\limits_{A(x_0,
\varepsilon, r_0)}
\Phi(Q(x))\frac{dm(x)}{\left(1+|x|^2\right)^n}\,,$$
where $\beta(x_0)=\left(1+(r_0+|x_0|)^2\right)^n.$ Therefore,
$$M_*\left(\varepsilon/r_0\right)\leqslant \frac{2\beta(x_0)}{\Omega_n r^n_0}M_0$$
for $\varepsilon\leqslant r_0 /\sqrt[n]{2},$ where $M_0$ is a
constant in~(\ref{eq5A}).
Observe that
$$M_*\left(\varepsilon/r_0\right)>\Phi(0)>0\,,$$
because $\Phi$ is increasing. Now, by~(\ref{eq34}) and~(\ref{eq35})
we obtain that
\begin{equation}\label{eq12}\int\limits_{\varepsilon}^{r_0}\frac{dr}{r^{\frac{n-1}{p-1}}q^{\frac{1}{p-1}}_{x_0}(r)}
\geqslant
\frac{1}{n}\int\limits_{\frac{2\beta(x_0)M_0e}{\Omega_nr^n_0}}
^{\frac{\Phi(0)r^n_0}{\varepsilon^n}}\frac{d\tau}
{\tau\left[\Phi^{\,-1}(\tau)\right]^{\frac{1}{p-1}}}\,.
\end{equation}
The desired conclusion follows from~(\ref{eq12})
and~(\ref{eq2}).~$\Box$
\end{proof}

\medskip
Recall that a pair $E=\left(A,\,C\right),$ where $A$ is an open set
in ${\Bbb R}^n,$ and $C$ is a compact subset of $A,$ is called {\it
condenser} in ${\Bbb R}^n$. Given $p\geqslant 1,$ a quantity
$${\rm cap}_p\,E={\rm
cap}\,(A,\,C)=\inf\limits_{u\,\in\,W_0\left(E\right)
}\quad\int\limits_{A}\,|\nabla u|^p\,\,dm(x)\,,
$$
where $W_0(E)=W_0\left(A,\,C\right)$ is a family of all nonnegative
absolutely continuous on lines (ACL) functions $u:A\rightarrow {\Bbb
R}$ with compact support in $A$ and such that $u(x)\geqslant 1$ on
$C,$ is called {\it $p$-capacity} of the condenser $E$. We write
${\rm cap}\,E$ for ${\rm cap}_n\,E.$ We also need the following
statement given in \cite[Proposition~II.10.2]{Ri}.

\medskip
\begin{proposition}\label{pr3}
{\sl\, Let $E=(A,\,C)$ be a condenser in ${\Bbb R}^n$ and let
$\Gamma_E$ be the family of all paths of the form
$\gamma:[a,\,b)\rightarrow A$ with $\gamma(a)\in C$ and
$|\gamma|\cap(A\setminus F)\ne\varnothing$ for every compact set
$F\subset A.$ Then ${\rm cap}\,E= M(\Gamma_E).$}
\end{proposition}

\medskip
In what follows, we set $a/\infty=0$ for $a\ne\infty,$ $a/0=\infty $
for $a>0$ and $0\cdot\infty =0.$ One of the most important
statements allowing us to connect the study of mappings
in~(\ref{eq2*!}) with the conditions~(\ref{eq2!!A})--(\ref{eq3!A})
is the following proposition. The principal points related to its
proof were indicated during the establishment of Lemma~1
in~\cite{SalSev$_2$}; however, for the sake of completeness of
presentation, we will establish it in full in the text.

\medskip
\begin{proposition}\label{pr2}{\sl\, Let $D$ be a domain in~${\Bbb R}^n$, $n\geqslant
2,$ let $x_0\in \overline{D}\setminus\{\infty\},$ let
$Q:D\rightarrow [0, \infty]$ be a Lebesgue measurable function and
let $f:D\rightarrow \overline{{\Bbb R}^n}$ be an open discrete
mapping satisfying relations ~(\ref{eq2!!A})--(\ref{eq3!A}) at a
point $x_0.$ If $0<r_1<r_2<\sup\limits_{x\in D}|x-x_0|,$ then
\begin{equation}\label{eq3B}
M_p(f(\Gamma(S(x_0, r_1), S(x_0, r_2), D)))\leqslant
\frac{\omega_{n-1}}{I^{p-1}}\,,
\end{equation}
where
\begin{equation}\label{eq9}
I=I(x_0,r_1,r_2)=\int\limits_{r_1}^{r_2}\
\frac{dr}{r^{\frac{n-1}{p-1}}q_{x_0}^{\frac{1}{p-1}}(r)}\,.
\end{equation}
If, in addition, $x_0\in D,$ $0<r_1<r_2<r_0={\rm dist}\,(x_0,
\partial D)$ and $E=\left(B(x_0, r_2), \overline{B(x_0, r_1)}\right),$ then
\begin{equation}\label{eq2A}
{\rm cap}_p\, f(E)\leqslant\ \frac{\omega_{n-1}}{I^{p-1}}\,,
\end{equation}
where $f(E)=\left(f\left(B(x_0, r_2)\right), f\left(\overline{B(x_0,
r_1)}\right)\right).$ }
\end{proposition}

\begin{proof} We may consider that $I \ne 0,$ since (\ref{eq3B}) and (\ref{eq2A}) are obvious,
in this case. We also may consider that $I\ne \infty.$ Otherwise, we
may consider $Q(x)+\delta$ instead of $Q(x)$ in (\ref{eq3B}) and
(\ref{eq2A}), and then pass to the limit as $\delta\rightarrow 0.$
Let $I\ne\infty.$

\medskip
Let us first prove relation~(\ref{eq3B}) for the case $x_0\in
\overline{D}\setminus\{\infty\}.$ Now $q_{x_0}(r)\ne 0$ for a.e.
$r\in(r_1,r_2).$ Set
$$ \psi(t)= \left \{\begin{array}{rr}
1/[t^{\frac{n-1}{p-1}}q_{x_0}^{\frac{1}{p-1}}(t)], & t\in (r_1,r_2)\
,
\\ 0,  &  t\notin (r_1,r_2)\ .
\end{array} \right. $$
In this case, by Fubini's theorem,
\begin{equation}\label{eq3}
\int\limits_{A} Q(x)\cdot\psi^p(|x-x_0|)\,dm(x)=\omega_{n-1} I\,,
\end{equation}
where $A=A(r_1,r_2, x_0)$ is defined in~(\ref{eq6}). Observe that a
function $\eta_1(t)=\psi(t)/I,$ $t\in (r_1,r_2),$
satisfies~(\ref{eq8B}) because
$\int\limits_{r_1}^{r_2}\eta_1(t)\,dt=1.$ Now, by the definition of
$f$ in~(\ref{eq2*!})
\begin{equation}\label{eq5G}
M_p(f(\Gamma(S(x_0, r_1), S(x_0, r_2), D)))\leqslant\int\limits_A
Q(x)\cdot {\eta_{1}}^p (|x-x_0|)\,dm(x)=
\frac{\omega_{n-1}}{I^{p-1}}\,.
\end{equation}
The first part of Proposition~\ref{pr2} is established. Let us prove
the second part, namely, relation~(\ref{eq2A}). Let $\Gamma_E$ and
$\Gamma_{f(E)}$ be families of paths in the sense of the notation of
Proposition~\ref{pr3}. By this proposition
\begin{equation}\label{eq6*!}
{\rm cap}_p\,f(E)={\rm cap}_p (f(B(x_0, r_2)), f(\overline{B(x_0
,r_1)}))=M_p(\Gamma_{f(E)})\,.
\end{equation}
Let $\Gamma^{*}$ be a family of all maximal $f$-liftings of
$\Gamma_{f(E)}$ starting in $\overline{B(x_0, r_1)}.$ Arguing
similarly to the proof of Lemma~3.1 in~\cite{Sev$_1$}, one can show
that $\Gamma^{*}\subset \Gamma_E.$ Observe that
$\Gamma_{f(E)}>f(\Gamma^{*}),$ and $\Gamma_E> \Gamma(S(x_0,
r_2-\delta), S(x_0, r_1), D)$ for sufficiently small $\delta>0.$
By~(\ref{eq5G}), we obtain that
$$
M_p(\Gamma_{f(E)})\leqslant M_p(f(\Gamma^{*}))\leqslant
M_p(f(\Gamma_E))\leqslant
$$
\begin{equation}\label{eq6B}
\leqslant M_p(f(\Gamma(S(x_0, r_1), S(x_0, r_2-\delta), A( r_1,
r_2-\delta, x_0))))\leqslant
\frac{\omega_{n-1}}{\left(\int\limits_{r_1}^{r_2-\delta}\frac{dt}{
t^{\frac{n-1}{p-1}}q_{x_0}^{\frac{1}{p-1}}(t)}\right)^{p-1}}\,.
\end{equation}
Observe that a function $\widetilde{\psi(t)}:=\psi|_{(r_1, r_2)} =
\frac{1}{t^{\frac{n-1}{p-1}}q_{x_0}^{\frac{1}{p-1}}(t)}$ is
integrable on $(r_1,r_2),$ because $I\ne \infty.$ Hence, by the
absolute continuity of the integral, we obtain that
\begin{equation}\label{eqA29}\int\limits_{r_1}^{r_2-\delta}\frac{dt}{t^{\frac{n-1}{p-1}}q_{x_0}
^{\frac{1}{p-1}}(t)}\rightarrow
\int\limits_{r_1}^{r_2}\frac{dt}{t^{\frac{n-1}{p-1}}q_{x_0}^{\frac{1}{p-1}}(t)}\end{equation}
as $\delta\rightarrow 0.$ By~(\ref{eq6B}) and (\ref{eqA29}), we
obtain that
\begin{equation}\label{eq7}
M_p(\Gamma_{f(E)})\leqslant
\frac{\omega_{n-1}}{\left(\int\limits_{r_1}^{r_2}\frac{dt}
{t^{\frac{n-1}{p-1}}q_{x_0}^{\frac{1}{p-1}}(t)}\right)^{p-1}}\,.
\end{equation}
Combining~(\ref{eq6*!}) and~(\ref{eq7}), we obtain~(\ref{eq2A}).
\end{proof}$\Box$

\medskip
The next lemma contains an application of the previous
Lemma~\ref{lem1} to mapping theory.

\medskip
\begin{lemma}\label{lem3}
{\sl\, Let $D$ be a domain in ${\Bbb R}^n,$ let $1\leqslant
p\leqslant n,$ let $\Phi:[0, \infty]\rightarrow [0, \infty] $ be a
strictly increasing convex function such that the relation and let
$x_0\in D.$ Denote by $\frak{R}_{\Phi, p}(D)$ the family of all
discrete open mappings for which there exists a Lebesgue measurable
function $Q=Q_f(x):{\Bbb R}^n\rightarrow [0, \infty],$ $Q(x)\equiv
0$ for $x\in {\Bbb R}^n\setminus D,$
satisfying~(\ref{eq2*!})--(\ref{eq8B}) for any $x_0\in D,$ and, in
addition, (\ref{eq5A}) holds for some $0<M_0<\infty.$ Let
$0<r_1<r_2<d_0={\rm dist}\,(x_0,
\partial D),$ and let $E=(B(x_0, r_2), \overline{B(x_0, r_2)})$ be a
condenser. If the relation~(\ref{eq2}) holds for some
$\delta_0>\tau_0:=\Phi(0),$ then
$${\rm cap}_p\,f(E)\rightarrow 0$$
as $r_1\rightarrow 0$ uniformly over $f\in \frak{R}_{\Phi, p}(D).$ }
\end{lemma}

\medskip
\begin{proof}
By Proposition~\ref{pr2}
\begin{equation}\label{eq6A}
{\rm cap}_p\,f(E)\leqslant \frac{\omega_{n-1}}{I^{p-1}}\,,
\end{equation}
where $\omega_{n-1}$ denotes an area of the unit sphere ${\Bbb
S}^{n-1}:=S(0, 1)$ in ${\Bbb R}^n,$
$I:=\int\limits_{r_1}^{r_2}\frac{dr}{r^{\frac{n-1}{p-1}}q^{\frac{1}{p-1}}_{x_0}(r)}$
and $q_{x_0}$ is defined in~(\ref{eq10}). The rest of the statement
follows from Lemma~\ref{lem1}.~$\Box$
\end{proof}

\section{Proof of Theorems~\ref{th1} and~\ref{th2}}

The following statement was proved for $p=n$ in
\cite[Lemma~3.11]{MRV$_2$} (see also \cite[Lemma~2.6, Ch.~III]{Ri}).

\medskip
\begin{proposition}\label{pr1}
{\sl\, Let $F$ be a compact proper subset of $\overline{{\Bbb R}^n}$
with ${\rm cap}\,F>0.$ Then for every $a>0$ there exists $\delta>0$
such that
$$
{\rm cap}\,(\overline{{\Bbb R}^n}\setminus F,\, C)\geqslant \delta
$$
for every continuum $C\subset \overline{{\Bbb R}^n}\setminus F$ with
$h(C)\geqslant a.$ }
\end{proposition}

{\it Proof of Theorem~\ref{th1}} largely uses the classical scheme
used in the quasiregular case, as well as applied by the author
earlier, see, for example, \cite[Theorem~4.1]{MRV$_2$},
\cite[Theorem~2.9.III]{Ri}, \cite[Theorem~8]{Cr},
\cite[Lemma~3.1]{Sev$_1$} and \cite[Lemma~4.2]{SalSev$_1$}.

\medskip
Let $x_0\in D,$ $\varepsilon_0<d(x_0, \partial D),$ and let $E=(A,
C)$ be a condenser, where $A=B(x_0, \varepsilon_0)$ and
$C=\overline{B(x_0, \varepsilon)}.$ As usual,
$\varepsilon_0:=\infty$ for $D={\Bbb R}^n.$ Let $a>0.$ Since ${\rm
cap}\,E>0,$ by Proposition~\ref{pr1} there exists
$\delta=\delta(a)>0$ such that
\begin{equation}\label{eq7A}{\rm cap}\,(\overline{{\Bbb R}^n}\setminus F,\,
E)\geqslant \delta
\end{equation}
for any continuum $C\subset \overline{{\Bbb R}^n}\setminus E$ such
that $h(C)\geqslant a.$ On the other hand, by Lemma~\ref{lem3} there
exists such that
$${\rm cap}\,f(E)\leqslant \alpha(\varepsilon)\,,\quad
\varepsilon\,\in (0,\,\varepsilon_0)\,,$$
for any $f\in{\frak F}^{\Phi}_{M_0, E}(D),$ where $\alpha$ is some
function such that $\alpha(\varepsilon)\rightarrow 0$ as
$\varepsilon\,\rightarrow 0.$ Now, for a number $\delta=\delta(a)$
there exists $\varepsilon_*=\varepsilon_*(a)$ such that
\begin{equation}\label{eq28*!}
{\rm cap\,} f(E)\leqslant \delta\,, \quad \varepsilon \in (0,
{\varepsilon_*}(a))\,.
\end{equation}
By~(\ref{eq28*!}), we obtain that
$${\rm cap\,}\left(\overline{{\Bbb R}^n}\setminus
E,\,f(\overline{B(x_0,\,\varepsilon)})\right)\leqslant {\rm
cap\,}\left(f({B(x_0,\varepsilon_0)}),
f(\overline{B(x_0,\,\varepsilon)})\right)\leqslant\delta$$
for $\varepsilon (0, \varepsilon_*(a)).$ Now, by~(\ref{eq7A}),
$h(f(\overline{B(x_0,\,\varepsilon)}))<a.$ Finally, for any $a>0$
there is $\varepsilon_*=\varepsilon_*(a)$ such that
$h(f(\overline{B(x_0,\,\varepsilon)}))<a$ for $\varepsilon\in (0,
\varepsilon_*(a)).$ Theorem is proved.~$\Box$

\medskip
To prove Theorem~\ref{th2}, we need the following most important
statement (see~\cite[(8.9)]{Ma}).

\medskip
\begin{proposition}\label{pr1A}
{\sl Given a condenser $E=(A, C)$ and  $1<p<n,$
$$
{\rm cap}_p\,E\geqslant n{\Omega}^{\frac{p}{n}}_n
\left(\frac{n-p}{p-1}\right)^{p-1}\left[m(C)\right]^{\frac{n-p}{n}}\,,
$$
where ${\Omega}_n$ denotes the volume of the unit ball in ${\Bbb
R}^n$, and $m(C)$ is the $n$-dimensional Lebesgue measure of $C.$}
\end{proposition}

\medskip
The basic lower estimate of capacity of a condenser $E=(A,
C)$ in ${\Bbb R}^n$ is given by
\begin{equation}\label{2.5} {\rm cap}_p\ E = {\rm cap}_p\ (A, C)
\geqslant \left(b_n\frac{(d(C))^p}
{(m(A))^{1-n+p}}\right)^{\frac{1}{n-1}}\,,\quad p>n-1,
\end{equation}
where $b_n$ depends only on $n$ and $p$ and $d(C)$ denotes the
diameter of $C$ (see \cite[Proposition~6]{Kr},
cf.~\cite[Lemma~5.9]{MRV$_1$}).

\medskip
{\it Proof of Theorem~\ref{th2}} is based on the approach used in
the proof of~Lemma~2.4 in~\cite{GSS}. Let $0<r_0<{\rm
dist\,}(x_0,\,\partial D).$ Consider a condenser $E=(A, C)$ with
$A=B(x_0, r_0),$ $C=\overline{B(x_0, \varepsilon)}.$ By
Lemma~\ref{lem3}, there is a function $\alpha=\alpha(\varepsilon)$
and $0<\varepsilon^{\,\prime}_0<r_0$ such that
$\alpha(\varepsilon)\rightarrow 0$ as $\varepsilon\rightarrow 0$
and, in addition,
$${\rm cap}_p\,f(E)\leqslant \alpha(\varepsilon)$$
for any $\varepsilon\in (0, \varepsilon^{\,\prime}_0)$ and
$f\in{\frak F}^{\Phi}_{M_0, p}(D).$
Applying Proposition~\ref{pr1A}, one obtains
$$
\alpha(\varepsilon)\geqslant{\rm cap}_p\,f(E)\geqslant
n{\Omega}^{\frac{p}{n}}_n
\left(\frac{n-p}{p-1}\right)^{p-1}\left[m(f(C))\right]^{\frac{n-p}{n}}\,,
$$
where ${\Omega}_n$ denotes the volume of the unit ball in ${\Bbb
R}^n,$ and $m(C)$ stands for the $n$-dimensional Lebesgue measure of
$C.$ In other words,
\begin{equation*}
m(f(C))\leqslant \alpha_1(\varepsilon)\,,
\end{equation*}
where $\alpha_1(\varepsilon)\rightarrow 0$ as
$\varepsilon\rightarrow 0.$ The last relation implies the existence
of a number $\varepsilon_1\in (0, 1),$ such that
\begin{equation}\label{eqroughb}
m(f(C))\leqslant 1\,,
\end{equation}
where $C=\overline{B(x_0, \varepsilon_1)}.$

\medskip
Further reasoning is related to the repeated application of
Lemma~\ref{lem3}. Consider one more condenser in this respect. Let
$E_1=(A_1, C_{\varepsilon}),$ $A_1=B(x_0, \varepsilon_1),$ and
$C_{\varepsilon}=\overline{B(x_0, \varepsilon)},$ $\varepsilon\in
(0, \varepsilon_1).$ By Lemma~\ref{lem3} there is a function
$\alpha_2(\varepsilon)$ and a number
$0<\varepsilon^{\,\prime}_0<\varepsilon_1$ such that
$$
{\rm cap}_p\,f(E_1)\leqslant \alpha_2(\varepsilon)
$$
for any $\varepsilon\in (0, \varepsilon^{\,\prime}_0),$ where
$\alpha_2(\varepsilon)\rightarrow 0$ as $\varepsilon\rightarrow 0.$
On the other hand, according to~(\ref{2.5}),
\begin{equation}\label{eq1} \left(c_1\frac{\left(d(f(\overline{B(x_0,
\varepsilon)}))\right)^p} {\left(m(f(B(x_0,
\varepsilon_1)))\right)^{1-n+p}}\right)^{\frac{1}{n-1}} \leqslant
{\rm cap\,}_p\,f\left(E_1\right)\leqslant \alpha_2(\varepsilon)\,.
\end{equation}
By~(\ref{eqroughb}) and~(\ref{eq1}), one gets
\begin{equation}\label{eq1A}
d(f(\overline{B(x_0, \varepsilon)})) \leqslant
\alpha_3(\varepsilon)\,,
\end{equation}
where $\alpha_3(\varepsilon)\rightarrow0$ as $\varepsilon\rightarrow
0.$ The proof of Theorem~\ref{th2} is completed, since the mapping
$f\in{\frak F}^{\Phi}_{M_0, p}(D)$ participating in~(\ref{eq1A}) is
arbitrary.~$\Box$

\section{Proof of Theorems~\ref{th3}--\ref{th6}}

The proofs of these theorems are conceptually close to the proofs of
Theorems 1–4 in~\cite{SevSkv$_1$} and use the same approach. Let's
start with the following very useful remark (see, for example,
\cite[Remark~1]{SevSkv$_1$}).

\medskip
\begin{remark}\label{rem1}
Let us show that, for a given domain $D_i,$ the
relation~(\ref{eq17***}) implies the so-called strong accessibility
of its boundary with respect to $p$-modulus (see
also~\cite[Theorem~6.2]{Na$_1$}). Let $i\in I,$ let $x_0\in
\partial D_i$ and let $U$ be some neighborhood of $x_0.$ We may assume
that $x_0\ne \infty.$ Let $\varepsilon_1> 0$ be such that $V:=B(x_0,
\varepsilon_1)$ and $\overline{V}\subset U.$ If $\partial U \ne
\varnothing $ and $\partial V \ne \varnothing,$ put
$\varepsilon_2:={\rm dist}\,(\partial U, \partial V)> 0.$ Let $F$
and $G$ be continua in $D_i$ such that $F\cap \partial U \ne
\varnothing \ne F \cap
\partial V$ and $G\cap \partial U \ne \varnothing \ne G \cap
\partial V. $ From the last relations it follows that $h(F)\geqslant
\varepsilon_2$ and $h(G)\geqslant \varepsilon_2.$ By the
equi-uniformity of $D_i$ with respect to $p$-modulus, we may find
$\delta=\delta(\varepsilon_2)> 0$ such that $M_p(\Gamma(F, G,
D_i))\geqslant \delta> 0.$ In particular, {\it for any neighborhood
$U$ of $x_0,$ there is a neighborhood $V$ of the same point, a
compact set $F$ in $D_i$ and a number $\delta> 0$ such that
$M_p(\Gamma(F, G, D_i))\geqslant \delta> 0 $ for any continuum
$G\subset D_i$ such that $G\cap \partial U \ne \varnothing \ne G
\cap \partial V.$} This property is called {\it a strong
accessibility} of $\partial D_i$ at the point $x_0$ with respect to
$p$-modulus. Thus, this property is established for any domain $D_i$
which is an element of some equi-uniform family $\{D_i\}_{i\in I}.$
\end{remark}

\medskip
{\it Proof of Theorem~\ref{th3}}. The equicontinuity of the family
$\frak{F}_{\Phi, A, p, \delta}(D)$ inside the domain $D$ follows
from \cite[Theorem~4.1]{RS} for $p=n$ and Theorem~\ref{th2} for
$p\ne n$. Put $f\in \frak{F}_{\Phi, A, p, \delta}(D)$ and
$Q=Q_f(x).$ Set
$$Q^{\,\prime}(x)=\begin{cases}Q(x), & Q(x)\geqslant 1\\ 1, & Q(x)<1\end{cases}\,.$$
Observe that $Q^{\,\prime}(x)$ satisfies (\ref{eq2!!A}) up to a
constant. Indeed,
$$
\int\limits_D\Phi(Q^{\,\prime}(x))\frac{dm(x)}{\left(1+|x|^2\right)^n}=
\int\limits_{\{x\in D: Q(x)< 1
\}}\Phi(Q^{\,\prime}(x))\frac{dm(x)}{\left(1+|x|^2\right)^n}+$$$$+
\int\limits_{\{x\in D: Q(x)\geqslant 1\}
}\Phi(Q^{\prime}(x))\frac{dm(x)}{\left(1+|x|^2\right)^n}\leqslant
M_0+\Phi(1)\int\limits_{{\Bbb
R}^n}\frac{dm(x)}{\left(1+|x|^2\right)^n}=M^{\,\prime}_0<\infty\,.$$
No, by~\cite[Theorem~2]{Sev$_2$} and Remark~\ref{rem1}, a mapping
$f\in\frak{F}_{\Phi, A, p, \delta}(D)$ has a continuous extension to
$\overline{D}$ for $p=n.$ In addition, by Lemma~\ref{lem1},
$\int\limits_{0}^{r_0}\frac{dt}{t^{\frac{n-1}{p-1}}q^{\,\prime\frac{1}{p-1}}_{x_0}(t)}=\infty,$
where $q^{\,\prime}_{x_0}(t)~=~\frac{1}{\omega_{n-1}r^{n-1}}
\int\limits_{S(x_0, t)}Q^{\,\prime}(x)\,d\mathcal{H}^{n-1}.$ In this
case, a continuous extension of the mapping from $f$ to $\partial D
$ can be established similarly to Theorem~1 in~\cite{Sev$_2$}. Note
that a rigorous proof of this fact was given
in~\cite[Theorem~1.2]{IS$_1$} for the case when the domains $D$ and
$f(D)$ have compact closures, and its proof in an arbitrary case can
be presented completely by analogy.

It remains to show that the family $\frak{F}_{\Phi, A, p,
\delta}(\overline{D})$ is equicontinuous at $\partial D.$ Suppose
the opposite. Then there is $x_0\in \partial D$ for which
$\frak{F}_{\Phi, A, p, \delta}(\overline{D})$ is not equicontinuous
at $x_0.$

Due to the additional application of the inversion $\varphi(x)=
\frac{x}{|x|^2},$ we may assume that $x_0\ne \infty.$ Then there is
a number $a>0$ with the following property: for any $m=1,2, \ldots$
there is $x_m \in \overline{D}$ and ${f}_m\in \frak{F}_{\Phi, A, p,
\delta}(\overline{D})$ such that $|x_0-x_m|< 1/m$ and, in addition,
$h(f_m(x_m), f_m(x_0))\geqslant a.$ Since $f_m$ has a continuous
extension at $x_0,$ we may find a sequence $x^{\,\prime}_m\in D,$
$x^{\,\prime}_m\rightarrow x_0$ as $m\rightarrow\infty$ such that
$h(f_m(x^{\,\prime}_m), f_m(x_0))\leqslant 1/m.$ Thus,
\begin{equation}\label{eq6***}
h(f_m(x_m), f_m(x^{\,\prime}_m))\geqslant a/2\qquad \forall\,\,m\in
{\Bbb N}\,.
\end{equation}
Since $f_m$ has a continuous extension to $\partial D,$ we may
assume that $x_m\in D.$ Since the domain $D$ is locally connected,
at the point $x_0,$ there is a sequence of neighborhoods $V_m$ of
the point $x_0$ with $h(V_m)\rightarrow 0$ for $m\rightarrow\infty$
such that the sets $D \cap V_m$ are domains and $D\cap V_m \subset
B(x_0, 2^{\,-m}).$ Without loss of the generality of reasoning,
going to subsequences, if necessary, we may assume that $x_m,
x^{\,\prime}_m \in D\cap V_m.$ Join the points $x_m$ and
$x^{\,\prime}_m$ by the path $\gamma_m:[0,1]\rightarrow {\Bbb R}^n$
such that $\gamma_m(0)=x_m,$ $\gamma_m(1)=x^{\,\prime}_m$ and
$\gamma_m(t)\in V_m$ for $t\in (0,1),$ see Figure~\ref{fig6}.
\begin{figure}[h]
\centerline{\includegraphics[scale=0.6]{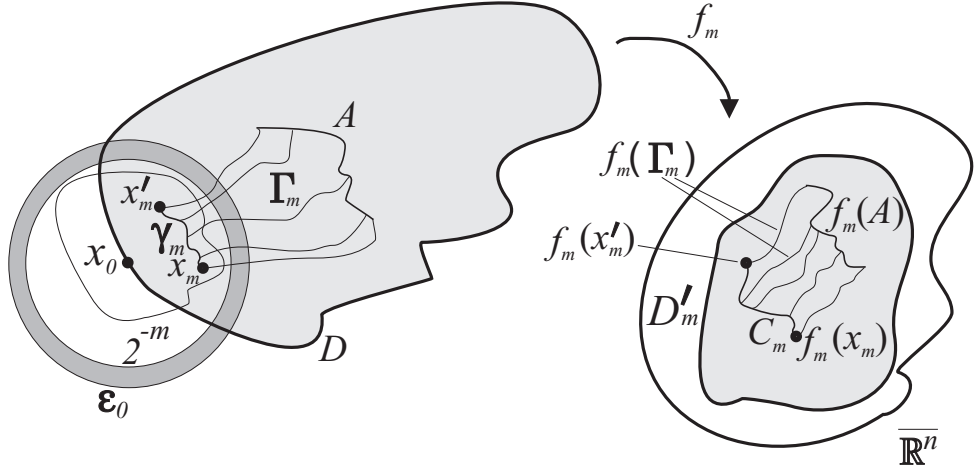}} \caption{To
the proof of Theorem~\ref{th3}}\label{fig6}
\end{figure}
We denote by $C_m$ the image of the path $\gamma_m (t)$ under the
mapping $f_m.$ From the relation~(\ref{eq6***}) it follows that
\begin{equation}\label{eq5.1}
h(C_m)\geqslant a/2\qquad\forall\, m\in {\Bbb N}\,,
\end{equation}
where $h$ denotes the chordal diameter of the set.

\medskip
Let $\varepsilon_0:={\rm dist}\,(x_0, A).$ Without loss of the
generality of reasoning, one may assume that the continuum $A$
participating in the definition of the class $\frak{F}_{\Phi, A, p,
\delta}(D),$ lies outside the balls $B(x_0, 2^{\,-m}),$
$m=1,2,\ldots, $ and $B(x_0, \varepsilon_0)\cap A=\varnothing.$ In
this case, the theorem on the property of connected sets that lie
neither inside nor outside the given set implies the relation
\begin{equation}\label{eq5B}
\Gamma_m>\Gamma(S(x_0, 2^{\,-m}), S(x_0, \varepsilon_0), D)\,,
\end{equation}
see e.g.~\cite[Theorem~1.I.5.46]{Ku}. Using Proposition~\ref{pr2}
and by~(\ref{eq4}), (\ref{eq5B}), we obtain that
\begin{equation}\label{eq10C}
M_p(f_m(\Gamma_m))\leqslant
\int\limits_{2^{\,-m}}^{r_0}\frac{dr}{r^{\frac{n-1}{p-1}}q^{\frac{1}{p-1}}_{mx_0}(r)}\rightarrow
0\,,\quad m\rightarrow\infty\,,
\end{equation}
where $q_{mx_0}(t)=\frac{1}{\omega_{n-1}r^{n-1}} \int\limits_{S(x_0,
t)}Q_m(x)\,d\mathcal{H}^{n-1}$ and $Q_m$ corresponds to the function
$Q$ of $f_m$ in~(\ref{eq2*!}).
On the other hand, observe that $f_m(\Gamma_m)=\Gamma(C_m, f_m(A),
D_m^{\,\prime}).$ By the condition of the lemma, $h(f_m(A))\geqslant
\delta$ for any $m\in {\Bbb N}. $ Therefore, by~(\ref{eq5.1})
$h(f_m(A))\geqslant \delta_1$ and $h(C_m)\geqslant\delta_1,$ where
$\delta_1:=\min\{\delta, a/2\}.$ Taking into account that the
domains $D_m^{\,\prime}:=f_m(D)$ are equ-uniform with respect to
$p$-modulus, we conclude that there exists $\sigma> 0$ such that

$$M_p(f_m(\Gamma_m))=M_p(\Gamma(C_m, f_m(A),
D_m^{\,\prime}))\geqslant \sigma\qquad\forall\,\, m\in {\Bbb N}\,,$$
which contradicts the condition~(\ref{eq6***}). The resulting
contradiction indicates that the assumption about the absence of
equicontinuity of $\frak {F}_{\Phi, A, p, \delta}(\overline {D})$
was wrong. The resulting contradiction completes the proof.~ $\Box$

\medskip
{\it Proof of Theorem~\ref{th4}.} The equicontinuity of the family
$\frak{R}_{\Phi, \delta, p, E}(D)$ inside the domain $D$ follows
from Theorem~\ref{th1} for $p=n$ and Theorem~\ref{th2} for $p\ne n$.
The possibility of continuous extension of any mapping $f\in
\frak{R}_{\Phi, \delta, p, E}(D)$ to $\partial D$ is established in
the same way as at the beginning of the proof of Theorem~\ref{th3},
and therefore the proof of this fact is omitted.

\medskip
It remains to show that the family $\frak{R}_{\Phi, \delta, p,
E}(D)$ is equicontinuous at $\partial D.$ Suppose the opposite. Then
there is $x_0\in \partial D$ for which $\frak{R}_{\Phi, \delta, p,
E}(D)$ is not equicontinuous at $x_0.$ Due to the additional
application of the inversion $\varphi(x)= \frac{x}{|x|^2},$ we may
assume that $x_0\ne \infty.$ Then there is a number $a>0$ with the
following property: for any $m=1,2, \ldots$ there is $x_m \in
\overline{D}$ and ${f}_m\in\frak{R}_{\Phi, \delta, p, E}(D)$ such
that $|x_0-x_m|< 1/m$ and, in addition, $h(f_m(x_m),
f_m(x_0))\geqslant a.$ Since $f_m$ has a continuous extension at
$x_0,$ we may assume that $x_m\in D.$ Besides that, we may find a
sequence $x^{\,\prime}_m\in D,$ $x^{\,\prime}_m\rightarrow x_0$ as
$m\rightarrow\infty$ such that $h(f_m(x^{\,\prime}_m),
f_m(x_0))\leqslant 1/m.$ Now, the relation~(\ref{eq6***}) holds.
Since the domain $D$ is locally connected, at the point $x_0,$ there
is a sequence of neighborhoods $V_m$ of the point $x_0$ with
$h(V_m)\rightarrow 0$ for $m\rightarrow\infty$ such that the sets $D
\cap V_m$ are domains and $D\cap V_m \subset B(x_0, 2^{\,-m}).$
Without loss of the generality of reasoning, going to subsequences,
if necessary, we may assume that $x_m, x^{\,\prime}_m \in D\cap
V_m.$ Join the points $x_m$ and $x^{\,\prime}_m$ by the path
$\gamma_m:[0,1]\rightarrow {\Bbb R}^n$ such that $\gamma_m(0)=x_m,$
$\gamma_m(1)=x^{\,\prime}_m$ and $\gamma_m(t)\in V_m$ for $t\in
(0,1),$ see Figure~\ref{fig2}.
\begin{figure}[h]
\centerline{\includegraphics[scale=0.6]{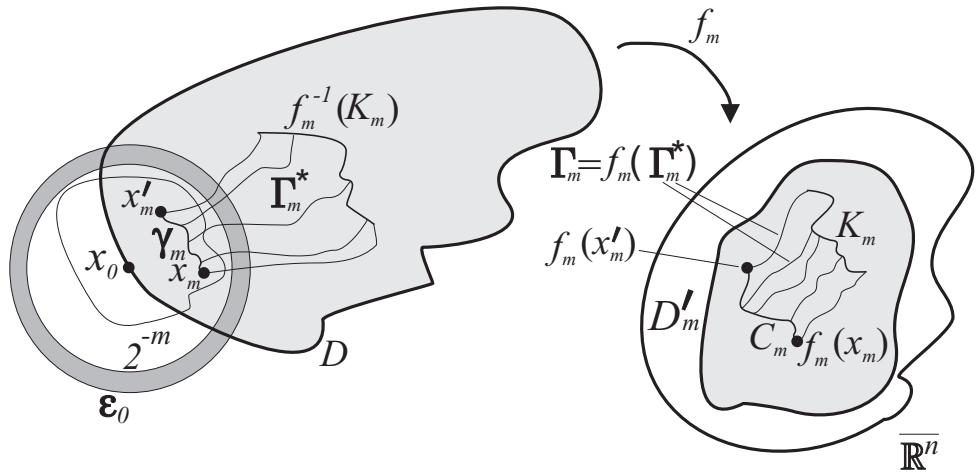}} \caption{To
the proof of Theorem~\ref{th4}}\label{fig2}
\end{figure}
We denote by $C_m$ the image of the path $\gamma_m $ under the
mapping $f_m.$ It follows from the relation~(\ref{eq6***}) that a
condition~(\ref{eq5.1}) is satisfied, where $h$ denotes a chordal
diameter of the set.

\medskip
By the definition of the family of mappings $\frak{R}_{\Phi, \delta,
p, E}(D),$ for any $m=1,2,\ldots ,$ any $f_m\in \frak{R}_{\Phi,
\delta, p, E}(D)$ and any domain $D^{\,\prime}_m:=f_m(D)$ there is a
continuum $K_m\subset D^{\,\prime}_m$ such that $h(K_m)\geqslant
\delta$ and $h(f_m^{\,-1}(K_m),
\partial D)\geqslant \delta>0.$ Since, by the hypothesis of the lemma, the domains $D^{\,\prime}_m$
are equi-uniform with respect to $p$-modulus, by~(\ref{eq5.1}) we
obtain that
\begin{equation}\label{eq13}
M_p(\Gamma(K_m, C_m, D^{\,\prime}_m))\geqslant b\,.
\end{equation}
for any $m=1,2,\ldots$ and some $b>0.$ Let $\Gamma_m$ be a family of
all paths $\beta:[0, 1)\rightarrow D^{\,\prime}_m$ such that
$\beta(0)\in C_m$ and $\beta(t)\rightarrow p\in K_m$ as
$t\rightarrow 1.$  Recall that a path $\alpha:[a, b)\rightarrow
{\Bbb R}^n$ is called a (total) $f$-lifting of a path $\beta:[a,
b)\rightarrow {\Bbb R}^n$ starting at $x_0,$ if $(f\circ
\alpha)(t)=\beta(t)$ for any $t\in [a, b).$ Let $\Gamma^*_m$ be a
family of all total $f_m$-liftings $\alpha:[0, 1)\rightarrow D$ of
$\Gamma_m$ starting at $\gamma_m.$ Such a family is well-defined
by~\cite[Theorem~3.7]{Vu}. Since the mapping $f_m$ is closed, we
obtain that $\alpha(t)\rightarrow f^{\,-1}_m(K_m)$ as $t\rightarrow
b-0,$ where $f^{\,-1}_m(K_m)$ denotes the pre-image of $K_m$ under
$f_m.$ Since $\overline{{\Bbb R}^n}$ is a compact metric space, the
set
$C_{\delta}:=\{x\in D: h(x, \partial D)\geqslant \delta\}$
is compact in $D$ for any $\delta>0$ and, besides that,
$f_m^{\,-1}(K_m)\subset C_{\delta}.$ By~\cite[Lemma~1]{Sm} the set
$C_{\delta}$ can be embedded in the continuum $E_{\delta}$ lying in
the domain $D.$ In this case, we may assume that ${\rm dist}\,(x_0,
E_{\delta})\geqslant \varepsilon_0$ by decreasing $\varepsilon_0.$
By the property of connected sets that lie neither inside nor
outside the given set, we obtain that
\begin{equation}\label{eq5C}
\Gamma^{\,*}_m>\Gamma(S(x_0, 2^{\,-m}), S(x_0, \varepsilon_0), D)\,,
\end{equation}
see e.g.~\cite[Theorem~1.I.5.46]{Ku}. Using Proposition~\ref{pr2}
and by~(\ref{eq4}), (\ref{eq5C}), we obtain that
$$M_p(f_m(\Gamma_m^*))\leqslant M_p(f_m(\Gamma(S(x_0, 2^{\,-m}),
S(x_0, \varepsilon_0), D)))\leqslant$$
\begin{equation}\label{eq10A}
\leqslant
\int\limits_{2^{\,-m}}^{r_0}\frac{dr}{r^{\frac{n-1}{p-1}}q^{\frac{1}{p-1}}_{mx_0}(r)}\rightarrow
0\,,\quad m\rightarrow\infty\,,
\end{equation}
where $q_{mx_0}(t)=\frac{1}{\omega_{n-1}r^{n-1}} \int\limits_{S(x_0,
t)}Q_m(x)\,d\mathcal{H}^{n-1}$ and $Q_m$ corresponds to $f_m$
in~(\ref{eq2*!}). Observe that $f_m(\Gamma^{\,*}_m)= \Gamma_m$ and
$M_p(\Gamma_m)=M_p(\Gamma(K_m, C_m, D^{\,\prime}_m)),$ so that
\begin{equation}\label{eq12B}
M_p(f_m(\Gamma^{\,*}_m))=M_p(\Gamma(K_m, C_m, D^{\,\prime}_m))\,.
\end{equation}
However, the relations (\ref{eq10A}) and (\ref{eq12B}) together
contradict~(\ref{eq13}). The resulting contradiction indicates that
the original assumption~(\ref{eq6***}) was incorrect, and therefore
the family of mappings $\frak{R}_{\Phi, \delta, p, E}(D)$ is
equicontinuous at every point $x_0\in \partial D.$~$\Box$

\medskip
{\it Proof of Theorem~\ref{th5}}. The equicontinuity of the family
$\frak{F}_{\Phi, A, p, \delta}(D)$ inside the domain $D$ follows
from~\cite[Theorem~4.1]{RS} for $p=n$ and Theorem~\ref{th2} for
$p\ne n$. The existence of a continuous extension of each
$f\in\frak{F}_{\Phi, A, p, \delta}(D)$ to a continuous mapping in
$\overline{D}$ follows from~\cite[Lemma~3]{Sev$_5$}. In particular,
the strong accessibility of $D_f^{\,\prime}=f(D)$ with respect to
$p$-modulus follows by Remark~\ref{rem1}.

\medskip
Let us show the equicontinuity of the family $\frak{F}_{\Phi, A, p,
\delta}(\overline{D})$ at $E_D,$ where $E_D$ denotes the space of
prime ends in $D.$ Suppose the contrary, namely, that the family
$\frak{F}_{\Phi, A, p, \delta}(\overline{D})$ is not equicontinuous
at some point $P_0\in E_D.$ Then there is a number $a>0,$ a sequence
$P_k\in \overline{D}_P,$ $k=1,2,\ldots,$ and elements
$f_k\in\frak{F}_{Q, A, p, \delta}(D)$ such that $d(P_k, P_0)<1/k$
and
\begin{equation}\label{eq3C}
h(f_k(P_k), f_k(P_0))\geqslant a\quad\forall\quad k=1,2,\ldots,\,.
\end{equation}
Since $f_k$ has a continuous extension to $\overline{D}_P,$ for any
$k\in {\Bbb N}$ there is $x_k\in D$ such that $d(x_k, P_k)<1/k$ and
$h(f_k(x_k), f_k(P_k))<1/k.$ Now, by~(\ref{eq3C}) we obtain that
\begin{equation}\label{eq4C}
h(f_k(x_k), f_k(P_0))\geqslant a/2\quad\forall\quad k=1,2,\ldots,\,.
\end{equation}
Similarly, since $f_k$ has a continuous extension to
$\overline{D}_P,$ there is a sequence $x_k^{\,\prime}\in D,$
$x_k^{\,\prime}\rightarrow P_0$ as $k\rightarrow \infty$ for which
$h(f_k(x_k^{\,\prime}), f_k(P_0))<1/k$ for $k=1,2,\ldots\,.$ Now, it
follows from~(\ref{eq4C}) that
\begin{equation}\label{eq5E}
h(f_k(x_k), f_k(x_k^{\,\prime}))\geqslant a/4\quad\forall\quad
k=1,2,\ldots\,,
\end{equation}
where $x_k$ and $x_k^{\,\prime}$ belong to $D$ and converge to $P_0$
as $k\rightarrow\infty,$ see Figure~\ref{fig3}.
\begin{figure}[h]
\centerline{\includegraphics[scale=0.6]{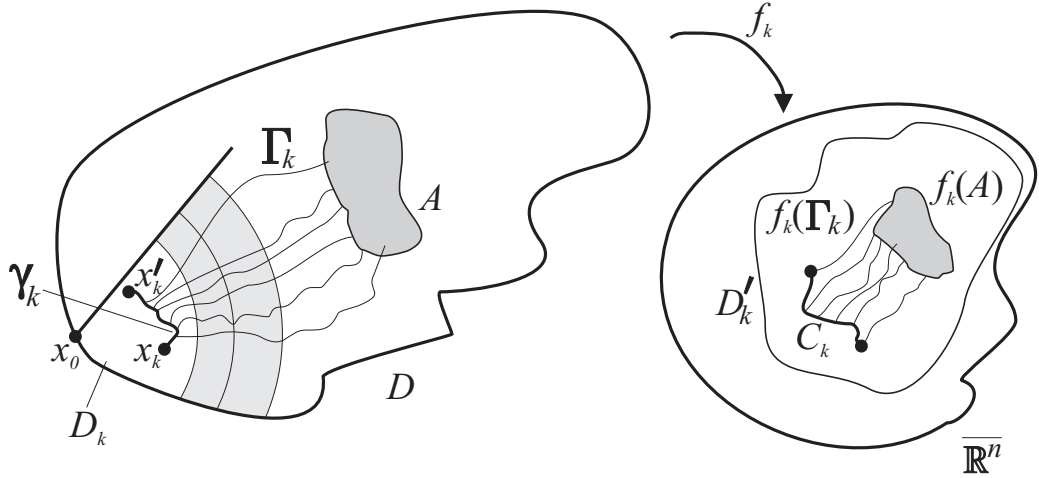}} \caption{To
the proof of Theorem~\ref{th5}}\label{fig3}
\end{figure}
By~\cite[Lemma~3.1]{IS$_2$}, cf.~\cite[Lemma~2]{KR}, a prime end
$P_0$ of a regular domain $D$ contains a chain of cuts $\sigma_k$
lying on spheres $S_k$ centered at some point $x_0\in \partial D $
and with Euclidean radii $r_k \rightarrow 0$ as $k\rightarrow \infty
$. Let $D_k$ be domains associated with the cuts $\sigma_k,$ $k=1,2,
\ldots .$ Since the sequences $x_k$ and $x_k^{\,\prime}$ converge to
the prime end $P_0$ as $k\rightarrow\infty,$ we may assume that
$x_k$ and $x_k^{\,\prime}\in D_k$ for any $k=1,2,\ldots, .$ Let us
join the points $x_k$ and $x_k^{\,\prime}$ by the path $\gamma_k,$
completely lying in $D_k.$ One can also assume that the continuum
$A$ from the definition of the class $\frak{F}_{\Phi, A, p,
\delta}(\overline{D})$ does not intersect with any of the domains
$D_k,$ and that ${\rm dist}\,(\partial D, A)>\varepsilon_0.$

\medskip
We denote by $C_k$ the image of the path $\gamma_k$ under the
mapping $f_k.$ It follows from the relation~(\ref{eq5E}) that
\begin{equation}\label{eq3G}
h(C_k)\geqslant a/4\qquad\forall\, k\in {\Bbb N}\,,
\end{equation}
where $h$ is a chordal diameter of the set.

\medskip
Let $\Gamma_k$ be a family of all paths joining $|\gamma_k|$ and $A$
in $D.$ By~\cite[Theorem~1.I.5.46]{Ku},
\begin{equation}\label{eq5D}
\Gamma_k>\Gamma(S(x_0, r_k), S(x_0, \varepsilon_0), D)\,.
\end{equation}
Using Proposition~\ref{pr2} and by~(\ref{eq4}), (\ref{eq5D}) we
obtain that
$$M_p(f_k(\Gamma_k))\leqslant M_p(f_k(\Gamma(S(x_0, r_k), S(x_0, \varepsilon_0), D)))\leqslant$$
\begin{equation}\label{eq14}
\leqslant
\int\limits_{r_k}^{r_0}\frac{dr}{r^{\frac{n-1}{p-1}}q^{\frac{1}{p-1}}_{kx_0}(r)}\rightarrow
0\,,\quad k\rightarrow\infty\,,
\end{equation}
where $q_{kx_0}(t)=\frac{1}{\omega_{n-1}r^{n-1}} \int\limits_{S(x_0,
t)}Q_k(x)\,d\mathcal{H}^{n-1}$ and $Q_k$ corresponds to the function
$Q$ of $f_k$ in~(\ref{eq2*!}).

\medskip
On the other hand, note that $f_k(\Gamma_k)=\Gamma(C_k, f_k(A),
D_k^{\,\prime}),$ where $D_k^{\,\prime}=f_k(D).$ Since by the
hypothesis of the lemma $h(f_k(A))\geqslant \delta$ for any $k\in
{\Bbb N},$  by~(\ref{eq3G}), $h(f_k(A))\geqslant \delta_1$ and
$h(C_k)\geqslant\delta_1,$ where $\delta_1:=\min\{\delta, a/4\}.$
Using the fact that the domains $D_k^{\,\prime}$ are equ-uniform
with respect to $p$-modulus, we conclude that there is $\sigma>0$
such that
$$M_p(f_k(\Gamma_k))=M_p(\Gamma(C_k, f_k(A),
D_k^{\,\prime}))\geqslant \sigma\qquad\forall\,\, k\in {\Bbb N}\,,$$
which contradicts condition~(\ref{eq14}). The resulting
contradiction indicates that the assumption of the absence of an
equicontinuity of the family $\frak{F}_{\Phi, A, p,
\delta}(\overline{D})$ was wrong. The resulting contradiction
completes the proof of the theorem.~$\Box$

\medskip
{\it Proof of Theorem~\ref{th6}.} The equicontinuity of the family
$\frak{R}_{\Phi, \delta, p, E}(D)$ inside the domain $D$ follows
from~\cite[Theorem~4.1]{RS} for $p=n$ and Theorem~\ref{th2} for
$p\ne n$. The existence of a continuous extension of each
$f\in\frak{R}_{\Phi, \delta, p, E}(D)$ to a continuous mapping in
$\overline{D}$ follows from~\cite[Lemma~3]{Sev$_5$}. In particular,
the strong accessibility of $D_f^{\,\prime}=f(D)$ with respect to
$p$-modulus follows by Remark~\ref{rem1}.

\medskip
It remains to show that the family $\frak{R}_{\Phi, \delta, p,
E}(D)$ is equicontinuous at $\partial_PD:=\overline{D}_P\setminus
D.$ Suppose the opposite. Arguing as in the proof of
Theorem~\ref{th5}, we construct two sequences $x_k$ and
$x_k^{\,\prime}\in D,$ converging to the prime end $P_0$ as
$k\rightarrow \infty,$ for which a relation~(\ref{eq5E}) holds. Let
us join the points $x_k$ and $x^{\,\prime}_k$ of the path
$\gamma_k:[0,1]\rightarrow{\Bbb R}^n$ such that $x_k^{\,\prime}\in
D,$ $\gamma_k(0)=x_k,$ $\gamma_k(1)=x^{\,\prime}_k$ and
$\gamma_k(t)\in D$ for $t\in (0,1).$  Denote by $C_k$ the image of
$\gamma_k $ under the mapping $f_k.$ It follows from the relation
(\ref{eq5E}) that
\begin{equation}\label{eq1B}
h(C_k)\geqslant a/4\,\,\forall\,\,k=1,2,\ldots .
\end{equation}
By~\cite[Lemma~3.1]{IS$_2$}, cf.~\cite[Lemma~2]{KR}, a prime end
$P_0$ of a regular domain $D$ contains a chain of cuts $\sigma_k$
lying on spheres $S_k$ centered at some point $x_0\in \partial D $
and with Euclidean radii $r_k \rightarrow 0$ as $k\rightarrow \infty
$. Let $D_k$ be domains associated with the cuts $\sigma_k,$ $k=1,2,
\ldots .$  Since the sequences $x_k$ and $x_k^{\,\prime}$ converge
to the prime end $P_0$ as $k\rightarrow\infty,$ we may assume that
$x_k$ and $x_k^{\,\prime}\in D_k$ for any $k=1,2,\ldots, .$

\medskip
By the definition of the family $\frak{R}_{\Phi, \delta, p, E}(D),$
for every $f_k\in \frak{R}_{Q, \delta, p, E}(D)$ and any domain
$D^{\,\prime}_k:=f_k(D)$ there is a continuum $K_k\subset
D^{\,\prime}_k$ such that $h(K_k)\geqslant \delta$ and
$h(f^{\,-1}(K_k),
\partial D)\geqslant \delta>0.$ Since, by the condition of the lemma, the
domains $D^{\,\prime}_k$ are equi-uniform with respect to
$p$-modulus, by~(\ref{eq1B}) we obtain that
\begin{equation}\label{eq13A}
M_p(\Gamma(K_k, C_k, D^{\,\prime}_k))\geqslant b\,.
\end{equation}
for any $k=1,2,\ldots$ and some $b>0.$ Let $\Gamma_k$ be a family of
all paths $\beta:[0, 1)\rightarrow D^{\,\prime}_k,$ where
$\beta(0)\in C_k$ and $\beta(t)\rightarrow p\in K_k$ as
$t\rightarrow 1.$ Let $\Gamma^*_k$ be a family of all total liftings
$\alpha:[0, 1)\rightarrow D$ of $\Gamma_k$ under the mapping $f_k$
starting at $\gamma_k.$ Such a family i well-defined
by~\cite[теорема~3.7]{Vu}. Since $f_k$ is closed,
$\alpha(t)\rightarrow f^{\,-1}_k(K_k)$ as $t\rightarrow 1,$ where
$f^{\,-1}_k(K_k)$ denotes the pre-image of $K_k$ under~$f_k.$ Since
$\overline{{\Bbb R}^n}$ is a compact metric space, the set
$C_{\delta}:=\{x\in D: h(x, \partial D)\geqslant \delta\}$
is compact in $D$ for any $\delta>0$ and, besides that,
$f_k^{\,-1}(K_k)\subset C_{\delta}.$ By~\cite[Lemma~1]{Sm} the set
$C_{\delta}$ can be embedded in the continuum $E_{\delta}$ lying in
the domain $D.$ In this case, we may assume that ${\rm dist}\,(x_0,
E_{\delta})\geqslant \varepsilon_0$ by decreasing $\varepsilon_0.$
Let $\Gamma_k$ be a family of all paths joining $|\gamma_k|$ and $A$
in $D.$ By~\cite[Theorem~1.I.5.46]{Ku},
\begin{equation}\label{eq5F}
\Gamma^*_k>\Gamma(S(x_0, r_k), S(x_0, \varepsilon_0), D)\,.
\end{equation}
Using Proposition~\ref{pr2} and by~(\ref{eq4}), (\ref{eq5F}) we
obtain that
$$M_p(f_k(\Gamma^*_k))\leqslant M_p(f_k(\Gamma(S(x_0, r_k), S(x_0, \varepsilon_0), D)))\leqslant$$
\begin{equation}\label{eq14A}
\leqslant
\int\limits_{r_k}^{r_0}\frac{dr}{r^{\frac{n-1}{p-1}}q^{\frac{1}{p-1}}_{kx_0}(r)}\rightarrow
0\,,\quad k\rightarrow\infty\,,
\end{equation}
where $q_{kx_0}(t)=\frac{1}{\omega_{n-1}r^{n-1}} \int\limits_{S(x_0,
t)}Q_k(x)\,d\mathcal{H}^{n-1}$ and $Q_k$ corresponds to the function
$Q$ of $f_k$ in~(\ref{eq2*!}). Observe that
$f_k(\Gamma^{\,*}_k)=\Gamma_k$ and, simultaneously,
$M_p(\Gamma_k)=M_p(\Gamma(K_k, C_k, D^{\,\prime}_k)).$ Now
\begin{equation}\label{eq12A}
M_p(f_k(\Gamma_k))=M_p(\Gamma(K_k, C_k, D^{\,\prime}_k))\,.
\end{equation}
Combining (\ref{eq14A}) and (\ref{eq12A}), we obtain a contradiction
with~(\ref{eq13A}). The resulting contradiction indicates that the
initial assumption~(\ref{eq6***}) was incorrect, and, therefore, the
family of mappings $\frak{R}_{\Phi, \delta, p, E}(D)$ is
equicontinuous at any point $x_0\in E_D.$~$\Box$

\medskip
\medskip
{\bf \noindent Evgeny Sevost'yanov} \\
{\bf 1.} Zhytomyr Ivan Franko State University,  \\
40 Bol'shaya Berdichevskaya Str., 10 008  Zhytomyr, UKRAINE \\
{\bf 2.} Institute of Applied Mathematics and Mechanics\\
of NAS of Ukraine, \\
1 Dobrovol'skogo Str., 84 100 Slavyansk,  UKRAINE\\
esevostyanov2009@gmail.com

\end{document}